\documentclass[10pt]{article}
\pdfoutput = 1
\usepackage{geometry}                % See geometry.pdf to learn the layout options. There are lots.
\usepackage{graphicx}
\usepackage{amssymb}
\usepackage{amsfonts}
\usepackage{amsmath}
\usepackage{aoa}
\newtheorem{theorem}{Theorem}
\newtheorem{proposition}{Proposition}

\newenvironment{proof}[1][Proof]{\smallbreak\noindent\textbf{#1.} }{\hfill\ \rule{0.5em}{0.5em}\par\smallbreak}
\def\Res{\hbox{Res}}
\newcommand{\Img}[2]{\includegraphics[width=#1truein]{#2}}

\title{Longest Run of Equal Parts in a Random Integer Composition}
\author{Ayla Gafni\footnote{This work was done during a summer internship at Algorithms Project, INRIA-Rocquencourt, F78153 Le Chesnay, France, May-July, 2009, under the direction of Philippe Flajolet.  Author's permanent email address is ayla.r.gafni@vanderbilt.edu.}
}
\date{July 31, 2009}                                           % Activate to display a given date or no date

\begin{document}
\maketitle
\begin{quote} \small
This note examines a problem in enumerative and asymptotic combinatorics involving the classical structure of integer compositions.  What is sought is an analysis on average and in distribution of the length of the longest run of consecutive equal parts in a composition of size $n$.  The problem was recently posed by Herbert Wilf (see arXiv: 0906.5196).  
\end{quote}
\normalsize
\vspace*{0.5in}

A \emph{composition} of an integer $n$ is a sequence $(x_1, \ldots, x_m)$ of positive integers such that 
$$n=x_1 + \cdots + x_m, \quad \text{and} \quad  x_i\ge 1.$$  
The $x_i$ are called the \emph{parts} and $n$ is the \emph{size} of the composition.  We wish to know the \emph{length of the longest run of equal parts} (which we denote by the random variable $L$) in a random composition of size $n$.  For instance, the composition 
$$3, 2, 1,  4, 4, 4, 4, 4, 7, 3, 5, 5, 4, 2,$$
has $L = 5$.  A composition with $L=1$ is known as a Carlitz composition.  The characteristics of Carlitz compositions and their generating function $C^{\left<1\right>}(z)$ (see Proposition~\ref{prop: gf}) are studied in great detail in \cite{Knopfmacher98carlitz, Louchard02probabilisticanalysis}.  The solution to the longest run problem can be broken down into four main sections.  In the first section, we find a family of generating functions for integer compositions that keeps track of the longest run of equal parts.  In the second section, we analyze the generating functions using singularity analysis to find an asymptotic estimate of the number of compositions of size $n$ with no run of length $k$.  In the third, we use that estimate to describe the probability distribution of the random variable $L$, and in the fourth, we calculate the mean and variance of the distribution.  The analysis here has some similarities to the analytic treatment of compositions in \cite{Archibald06distinctvalues, Knopfmacher98carlitz, Louchard02probabilisticanalysis}, and the methods and notation used in this note are detailed in the book  \emph{Analytic Combinatorics} by Flajolet and Sedgewick \cite{Flajolet09analyticCombinatorics}.  
This note was motivated by a question of Wilf, posed at the Analysis of Algorithms 09 Conference (Frejus, June 2009); see \cite{Wilf09longestrunlength}.  

The author would like to thank Herbert Wilf for suggesting this problem and Philippe Flajolet for his direction and support throughout this project.

\section{Enumerative Aspects of Compositions}

The enumeration of integer compositions is easily solved using basic combinatorics.  We can create a graphical model of a composition by representing the integers in unary using small discs (``$\bullet$'') and drawing bars between some of the balls.  The following is an example using the composition $2+3+1+1+3=10$: 
$$\bullet\bullet|\bullet\bullet\bullet|\bullet|\bullet|\bullet\bullet\bullet.$$ 
Using this ``balls-and-bars'' model, we see that the number of compositions of the integer $n$ is
$$C_n  = 2^{n-1},$$
since a composition can be viewed as the placement of separation bars at a subset of the $n-1$ spaces between the balls. 

We can also find the enumeration of compositions with the symbolic method \cite[p.~40]{Flajolet09analyticCombinatorics}.  If the integers are represented in unary, then the combinatorial class of positive integers ($\cls I $) can be thought of as a sequence of atoms ($\cls Z$) so that
$$\cls I  = \Seq[\ge 1]{\clsAtom} \quad \Longrightarrow \quad I(z) = \frac{z}{1-z}.$$
Since an integer composition is simply a sequence of positive integers, we can easily derive the generating function for the class \cls C of compositions from the specification
$$\cls C = \Seq{\cls I} \quad \Longrightarrow \quad C(z) = \frac{1}{1-I(z)} = \frac{1}{1-\frac{z}{1-z}} = \frac{1-z}{1-2z}.$$
Throughout this note, we let $[z^n] f(z)$ be the coefficient of $z^n$ in the expansion of $f(z)$ at $0$:
$$[z^n] \sum_n f_n z^n = f_n.$$
We find that our result using the symbolic method is consistent with the above combinatorial argument, since
$$[z^n]C(z) = [z^n]\frac{1}{1-2z} - [z^n]\frac{z}{1-2z} = 2^n - 2^{n-1} = 2^{n-1}.$$

Now that we have a generating function for \emph{all} integer compositions, we need another generating function for compositions, which keeps track of the longest run of equal parts.  We begin by examining Smirnov words, i.e., words over an $m$-ary alphabet such that no letter occurs twice in a row.  Words over the $m$-ary alphabet $\{a_1,\ldots, a_m\}$ can be represented by the multivariate generating function
$$W(x_1,\ldots,x_m) = \frac{1}{1-(x_1+\cdots + x_m)},$$
where $x_j$ marks the number of times the letter $a_j$ occurs in a word.  That is, the expression $[x_1^{n_1},\ldots, x_m^{n_m}]W(x_1,\ldots,x_m)$ denotes the number of words in which the letter $a_1$ occurs $n_1$ times, $a_2$ occurs $n_2$ times, and so on. 

Similarly, let $S(y_1,\ldots, y_m)$ be the generating function for Smirnov words, where $y_j$ marks the number of times the letter $a_j$ occurs in a word.  Now, given a Smirnov word, one can obtain \emph{any} word by replacing $a_j$ with a nonempty sequence of $a_j$ (i.e., $a_j\times \Seq{a_j}$).  In terms of generating functions, this translates to
$$W(x_1, \ldots, x_m) = S\left(\frac{x_1}{1-x_1}, \ldots , \frac{x_m}{1-x_m}\right).$$
We use this to find the generating function for Smirnov words in terms of the generating function for all words:
$$S(y_1, \ldots, y_m) = W\left(\frac{y_1}{1+y_1}, \ldots , \frac{y_m}{1+y_m}\right),$$
so that we have
$$S(y_1, \ldots, y_m) = \left(1-\sum_{j=1}^m\frac{y_j}{1+y_j}\right)^{-1}.$$

We would like to generalize $S(y_1, \ldots, y_m) $ to the generating function $S^{\left<k\right>}(y_1, \ldots, y_m)$ for words over an $m$-ary alphabet such that no letter occurs $k$ times in a row.  We can obtain this via the substitution
$$y_j \rightarrow \sum_{i=1}^{k-1}y_j^i = \frac{y_j-y_j^k}{1-y_j},$$
which yields
$$S^{\left<k\right>}(y_1, \ldots, y_m) = \left(1-\sum_{j=1}^m \frac{y_j-y_j^k}{1-y_j^k}\right)^{-1}.$$

Integer compositions are sequences of positive integers, and hence can be thought of as ``words'' over an infinite alphabet where the ``letters'' are the positive integers.  So, by letting $m$ tend to infinity and replacing $y_j$ with $z^j$, we obtain the generating function for integer compositions with no part appearing $k$ times in a row.

\begin{proposition} The generating function for integer compositions with no part appearing $k$ times consecutively is given by
$$C^{\left<k\right>}(z) = \left(1-\sum_{j=1}^\infty z^j\frac{1-z^{j(k-1)}}{1-z^{jk}}\right)^{-1}.$$
\label{prop: gf}
\end{proposition}

We have now finished the first step of the solution, which is to find a family of generating functions for integer compositions, indexed by their longest run of equal parts.  Our result is essentially equivalent to one given by Wilf (Theorem 3 of \cite{Wilf09longestrunlength}).  Wilf obtains it by means of the correlation theory of Guibas-Odlyzko.  For similar problems, Knopfmacher and Prodinger employ the technique of ``adding a slice'' \cite{Knopfmacher98carlitz}.  Our derivation above, based on Smirnov words, has the advantage of great versatility, and it follows \cite[p.~205]{Flajolet09analyticCombinatorics}.

\section{Singularity Analysis}

We move on to the second step of the solution, which is to view $C^{\left<k\right>}(z)$ as a function in the complex plane and perform singularity analysis.  This step is much more technical than the first. 
The important fact used here is that information about the function's Taylor coefficients is contained in the asymptotic behavior of the function at its singularities.  In fact, most of the relevant information is contained in the dominant singularity.\footnote{A dominant singularity is one of smallest modulus.}  Thus, our next step is to find the dominant singularity of the function $C^{\left<k\right>}(z)$ and show that it is an \emph{isolated pole}.  This is established in the following two propositions.  The analysis here is reminiscent of that of longest runs in binary words in \cite[p.~308]{Flajolet09analyticCombinatorics}.  However, the case of integer compositions is somewhat more complicated because we do not have \emph{rational} generating functions, and hence encounter additional difficulties in finding the dominant singularity of $C^{\left<k\right>}(z)$.

\begin{proposition} The dominant singularity of the generating function $C^{\left<k\right>}(z)$ satisfies
$$\rho_k = \frac{1}{2}\left(1 + 2^{-k-2} + O(k2^{-2k})\right), \qquad k\rightarrow \infty.$$
\label{prop: rho}
\end{proposition}

\begin{proof}
We consider the denominator of $C^{\left<k\right>}(z)$:
\begin{align*} 1-\sum_{j=1}^\infty z^j\frac{1-z^{j(k-1)}}{1-z^{jk}} & 
= 1 - \sum_{j=1}^\infty z^j + \sum_{j=1}^\infty z^j - \sum_{j=1}^\infty z^j\frac{1-z^{j(k-1)}}{1-z^{jk}} \\ &
= 1 - \sum_{j=1}^\infty z^j + \sum_{j=1}^\infty z^{jk} \frac{1-z^{j}}{1-z^{jk}}.
\end{align*}
Let $h(z) = \sum_{j \ge 1} z^j$ and $g(z) =  \sum_{j\ge 1} z^{jk} \frac{1-z^{j}}{1-z^{jk}}$.  Let $\rho_k$ be the dominant positive singularity of $C^{\left<k\right>}(z)$, whose existence is guaranteed by Pringsheim's Theorem \cite[p.~240]{Flajolet09analyticCombinatorics}.  The value $\rho_k$ is the solution to the equation 
$$h(z) - g(z) = 1.$$

Since $h(z) = z/(1-z)$, we have that $h(1/2) = 1$.  Also, $g(z)$ is positive when $z$ is real between $0$ and $1$.  Therefore, $h(1/2) - g(1/2) <1$.  Simple calculations show that $h(0.6) - g(0.6) >1$ for $k \ge 2$.  Hence we have that 
$$\frac{1}{2} < \rho_k < \frac{3}{5}.$$
To get a closer estimate for $\rho_k$, we consider that $\rho_k$ is the fixed point of the equation
\begin{equation}
z = h^{-1}(1+g(z)) = \frac{1+g(z)}{2+g(z)},
\label{eq: fixed point}
\end{equation}
and use iterative methods to estimate the fixed point.  We let
$$z_0=\frac{1}{2} \qquad \text{and} \qquad z_{i+1} = \frac{1+g(z_i)}{2+g(z_i)}.$$
Since $1/2<\rho_k$, we have that $z_1<\rho_k$, and by induction $z_{i-1} < z_i<\rho_k$ for all $i>0$.  Since $\rho_k$ is the unique fixed point of equation (\ref{eq: fixed point}), the sequence $\{z_i\}_{i=0}^\infty$ must converge to $\rho_k$.

Before computing $z_1$, it is helpful to simplify our definition of $g(z)$:
$$g(z) = \sum_{j\ge 1} z^{jk} \frac{1-z^{j}}{1-z^{jk}} = z^k\frac{1-z}{1-z^k} + O(z^{2k})$$
Now, we have that
\begin{align*}
z_1 & = \frac{1+g(1/2)}{2+g(1/2)} = \frac{1 + 2^{-k-1}(1-2^{-k})^{-1} + O(2^{-2k})}{2 + 2^{-k-1}(1-2^{-k})^{-1} + O(2^{-2k})} \\ &
= \frac{1}{2} + 2^{-k-3} + O(2^{-2k})
\end{align*}
Further iterations will increase our estimate by adding terms which are  $O(k2^{-2k})$, so we have that $\rho_k - z_1 = O(k2^{-2k}).$  Therefore,
$$\rho_k =z_1 + O(k2^{-2k}) =  \frac{1}{2}\left(1 + 2^{-k-2} + O(k2^{-2k})\right),$$
as desired.
\end{proof}

\begin{proposition}
For $k \ge 4$, the value $\rho_k$ is the only singularity of the function $C^{\left<k\right>}(z)$
in the domain $|z|<3/5$.
\end{proposition}

\begin{proof}
Let 
$$f(z) = 1-h(z), \quad \text{where  } \ h(z) = \frac{z}{1-z},$$ and let 
$$ g(z) =  \sum_{j\ge 1} z^{jk} \frac{1-z^{j}}{1-z^{jk}}.$$  
Notice that 
$$C^{\left<k\right>}(z) = \frac{1}{1-h(z) + g(z)} = \frac{1}{f(z) + g(z)},$$ 
Thus, once we have shown that $|g(z)| \le |f(z)|$ for all $z$ on the circle $|z| = 3/5$, then by Rouch\'e's Theorem \cite[p.~270]{Flajolet09analyticCombinatorics}, $f(z)$ and $f(z) + g(z)$ will have the same number of zeros in the domain $|z| < 3/5$.  The proposition will follow immediately, since there is only one root of $f(z)$ in that domain, which is $z=1/2$.  

To show that $|g(z)| \le |f(z)|$ for all $z$ on the circle $|z| = 3/5$, we  first bound $|g(z)|$ from above and then bound $|f(z)|$ from below.  On this circle, we have that
$$\left|\frac{1}{1-z^{jk}}\right|  \le \frac{1}{1-|z^{jk}|} \le \frac{1}{1-0.6^{k}} \qquad\hbox{and}\qquad
|1-z^j| \le 1 + |z^j| \le 1.6,$$
which implies
$$|g(z)| \le \sum_{j\ge 1} |z^{jk}| \frac{|1-z^{j}|}{|1-z^{jk}|}
\le \frac{1.6}{1-0.6^{k}}\sum_{j\ge 1}|z^{jk}|
\le \frac{(1.6)(0.6)^k}{(1-0.6^{k})^2}.$$
For $k\ge4$, this becomes
$$|g(z)| \le \frac{(1.6)(0.6)^4}{(1-0.6^{4})^2} \approx 0.2737.$$

We now need to bound $|1-h(z)|$ from below on the circle $|z|=3/5$.  That is, we need to find the distance between the point $1$ and the image of the circle $|z|=3/5$ under the linear fractional transformation
$$\phi(z)  = \frac{z}{1-z}.$$
The image of a circle under a linear fractional transformation is again a circle.  Thus, since $\phi(-0.6)=-0.375$ and $\phi(0.6) = 1.5$, we have that the image of $\{|z|=3/5\}$ under $\phi$ is the circle $|z-0.5625| = 0.9375$, which comes closest to the point $1$ on the positive real axis.  Therefore,
$$|f(z)| = |1-h(z)| \ge |1-1.5| = 0.5.$$

Therefore, for $|z|=3/5$,  
$$|g(z)| \le 0.274 < 0.5 \le |f(z)|,$$
so Rouch\'e's Theorem can be applied.  Thus there is only one root of  $f(z) + g(z)$ in the given domain, and that root must be $\rho_k$.
\end{proof}

Now that we have found and isolated the dominant pole, $\rho_k$, we use the Residue Theorem to extract information about the coefficients $C_n^{\left<k\right>}$ from the behavior of $C^{\left<k\right>}(z)$ at $\rho_k$.  This is our main approximation, expressed primarily as $n\rightarrow\infty$, but also allowing for $k$ to get large.

\begin{proposition}
The number of compositions of $n$ with no run of $k$ equal parts satisfies 
$$C_n^{\left<k\right>} = \rho_k^{-n-1}(1-\rho_k)^2\left(1+\epsilon(k)\right) + O\left(\left(\frac{5}{3}\right)^n\right), \qquad n\rightarrow\infty,$$
uniformly with respect to $k \ge 4$, where $\epsilon(k) = O(k2^{-k})$ as $k\rightarrow\infty$.
\label{prop: number of comp}
\end{proposition}

\begin{proof}
The Residue Theorem gives
\begin{align}
\frac{1}{2\pi i} \int_{|z|=3/5} \frac{C^{\left<k\right>}(z)}{z^{n+1}} dz &
= \Res\left(\frac{C^{\left<k\right>}(z)}{z^{n+1}} ; z = 0 \right) + \Res\left(\frac{C^{\left<k\right>}(z)}{z^{n+1}} ; z = \rho_k \right) \nonumber \\ &
= C_n^{\left<k\right>} + \Res\left(\frac{C^{\left<k\right>}(z)}{z^{n+1}} ; z = \rho_k \right). \label{eq: res thm}
\end{align}
On the other hand, previous arguments have shown
$$ |1-h(z) + g(z)| \ge |1-h(z)| - |g(z)| \ge 0.5 - 0.2737 > 0.22.$$
Hence
\begin{equation}
\frac{1}{2\pi i} \int_{|z|=3/5} \frac{C^{\left<k\right>}(z)}{z^{n+1}} dz  = \frac{1}{2\pi i} \int_{|z|=3/5} \frac{1}{1-h(z) + g(z)}\frac{dz}{z^{n+1}} <  \frac{5}{2\pi i}\left(\frac{3}{5}\right)^{-n-1}. \label{eq: error bound}
\end{equation}
Combining equations (\ref{eq: res thm}) and (\ref{eq: error bound}), we obtain
$$C_n^{\left<k\right>} = - \Res\left(\frac{C^{\left<k\right>}(z)}{z^{n+1}} ; z = \rho_k \right) +  O\left(\left(\frac{5}{3}\right)^n\right).$$

We now need to find an estimate for the residue $R_{n,k} := - \Res\left(\frac{C^{\left<k\right>}(z)}{z^{n+1}} ; z = \rho_k \right)$.  We have
$$\frac{C^{\left<k\right>}(z)}{z^{n+1}} = \frac{z^{-n-1}}{D_k(z)}, \quad \text{where} \quad 
D_k(z) = 1 - \frac{z}{1-z} + \sum_{j=1}^\infty z^{jk} \frac{1-z^{j}}{1-z^{jk}}.$$  
Hence 
$$R_{n,k} = - \frac{\rho_k^{-n-1}}{D_{k}'(\rho_k)}.$$
A straightforward calculation gives
$$D_{k}'(z) = \frac{d}{dz} D_k(z) = -\frac{1}{(1-z)^2} + (k+1)z^k + kz^{k-1} + O(kz^{2k}),$$
uniformly for $z$ near $\rho_k$ as $k\rightarrow\infty$. 
So we have
\begin{align*}
\frac{1}{D_{k}'(\rho_k)} & 
= -\frac{1}{ \frac{1}{(1-\rho_k)^2} - (k+1)\rho_k^k - k\rho_k^{k-1} + O(k\rho_k^{2k})} \\ &
= -(1-\rho_k)^2 - (k+1)\rho_k^k - k\rho_k^{k-1} + O(k^2 \rho_k^{2k}).
\end{align*}
We thus obtain
$$R_{n,k}  =  \rho_k^{-n-1}(1-\rho_k)^2\left(1+O(k2^{-k})\right),$$
from which the result follows immediately.
\end{proof}

This concludes the second step of the solution.  In the remainder of the note, we find the asymptotic form of the probability distribution of the random variable $L$ and estimate its mean and variance.

\section{Analysis of the Probability Distribution}

In this section we exploit the main approximations of Propositions \ref{prop: rho} and \ref{prop: number of comp} to describe the probability distribution of the random variable $L$.  This development is in the scale of $\lg n$, which is anticipated on probabilistic grounds (see Conclusion).  Our main goal here is to show a double exponential form for the bulk of the distribution.

\begin{theorem}
Let $L$ be the random variable measuring the longest run of equal parts in a random integer composition of size $n$.  Then $L$ satisfies
\footnote{In this note, $\mathbb{P}_n$ refers to the probabilistic model where all compositions of size $n$ are taken equally likely.}
$$\mathbb{P}_n(L < k) = e^{-n/2^{k+2}} \left(1+O\left(\frac{\log n}{\sqrt{n}}\right)\right),$$
uniformly for $k\in\mathbb{Z}$ such that $\frac{3}{4}\lg n \le k \le 2\lg n$.  Equivalently
\footnote{Here $\lg n \equiv \log_2 n$ and $\{\alpha\}$ represents  the fractional part of $\alpha$ (i.e., $\{\alpha\} = \alpha - \lfloor \alpha \rfloor$).}
$$\mathbb{P}_n(L< \lfloor \lg n \rfloor + h) = e^{-\omega(n)2^{-h-2}}\left(1+O\left(\frac{\log n}{\sqrt{n}}\right)\right), \qquad \omega(n) = 2^{\{\lg n\}},$$
uniformly for $h\in\mathbb{Z}$ such that $-\frac{3}{4}\lg n + \{ \lg n \} \le h \le \lg n + \{ \lg n \}$.
\label{thm: prob}
\end{theorem}

\noindent The formulae above do not represent a single distribution, but rather a \emph{family} of distributions indexed by the fractional part of $\lg n$.  The second form given above shows explicitly how the behavior of the distribution depends on the value of $n$ in relation to powers of $2$.  Figure~\ref{fig: histogram} displays the histograms of the exact distributions of $L$ for $n=20, 40, \ldots 500$.  We see that the peak of the distribution does not increase smoothly with $n$, but instead incurs jumps and irregularities, which are a result of the distribution's dependence on $\{\lg n\}$.  Such a phenomenon is analogous to that found in the distribution of longest runs in binary words \cite[p.~308]{Flajolet09analyticCombinatorics}.

\begin{figure}

\begin{center}
\Img{2.0}{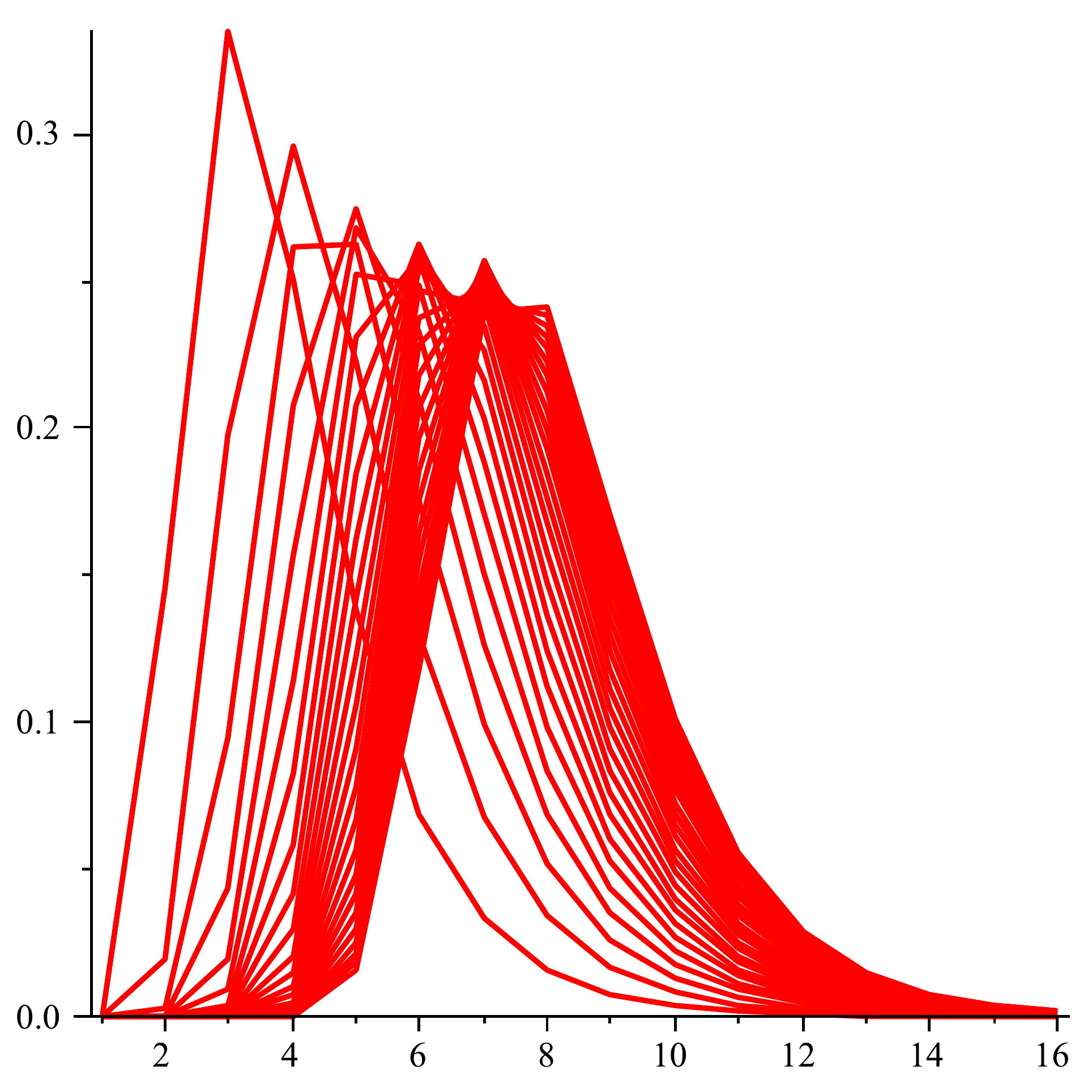}
\end{center}

\caption{\label{fig: histogram} 
Histograms of the exact distributions $\mathbb{P}_n(L)$ for $n = 20, 40, \ldots, 500$.
}

\end{figure}

\begin{proof}
Recall that the total number of compositions of size $n$ is $C_n = 2^{n-1}$.  Therefore, by Proposition~\ref{prop: number of comp},
\begin{equation}
\mathbb{P}_n(L< k) = \frac{C_n^{\left<k\right>}}{C_n} = \frac{2}{\rho_k}(2\rho_k)^{-n}(1-\rho_k)^2\left(1+\epsilon(k)\right) + O\left(\left(\frac{5}{6}\right)^n\right),
\label{eq: prob L<k} \end{equation}
where $\epsilon(k) = O(k2^{-k})$.  Note that in the region $\frac{3}{4}\lg n \le k \le 2\lg n$,  the value of $(2\rho_k)^{-n}$ satisfies
\begin{equation}
(2\rho_k)^{-n} = \exp(-n \log (2\rho_k)) = \exp \left(-\frac{n}{2^{k+2}} + O\left(\frac{n}{2^{\frac{3}{2} \lg n}}\right)\right)
= e^{-n/2^{k+2} }\left(1 + O\left(\frac{1}{\sqrt{n}}\right)\right),
\label{eq; rho^n} \end{equation}
and the value of the coefficient $\frac{2}{\rho_k}(1-\rho_k)^2$ is of the form
$$\frac{2}{\rho_k}(1-\rho_k)^2 = 1 + O(2^{-k}) = 1 + O\left(\frac{1}{n^{3/4}}\right).$$
Hence, in this region, 
\begin{align}
\mathbb{P}_n(L < k) & = e^{-n/2^{k+2} }\left(1 + O\left(\frac{1}{\sqrt{n}}\right)\right)\left(1 + O\left(\frac{1}{n^{3/4}}\right)\right)\left(1+O\left(\frac{\lg n}{n^{3/4}}\right)\right) + O\left(\left(\frac{5}{6}\right)^n\right)\nonumber \\ &
=  e^{-n/2^{k+2} }\left(1 + O\left(\frac{\log n}{\sqrt{n}}\right)\right). 
\label{eq: result with k}
\end{align}

We now adjust equation (\ref{eq: result with k}) to reflect the distribution's dependence on the placement of $n$ with respect to powers of $2$.  The floor function is used to emphasize the fact that $k$ must be an integer.
  Let $k = \lg n + x$ where $x \in [-\frac{3}{4}\lg n, \lg n]$.  Then we can write
$$k = \lfloor \lg n \rfloor + \{ \lg n \} + x,$$ 
and we let $h =  \{ \lg n \} + x$.  We require that $h$ be an integer and that $h - \{ \lg n \} \in [-\frac{3}{4}\lg n, \lg n]$.  In other words
$$k =  \lfloor \lg n \rfloor + h, \text{ where } h\in\mathbb{Z}, \text{ such that }  -\frac{3}{4}\lg n + \{ \lg n \} \le h \le \lg n + \{ \lg n \}.$$
Inserting this into equation (\ref{eq: result with k}) we obtain
\begin{align*}
\mathbb{P}_n(L < \lfloor \lg n \rfloor + h) & 
= \exp\left(\frac{-n}{2^{ \lg n - \{ \lg n \} + h+2}}\right)\left(1 + O\left(\frac{\log n}{\sqrt{n}}\right)\right) \\ &
= e^{-2^{\{ \lg n \} - h-2}}\left(1 + O\left(\frac{\log n}{\sqrt{n}}\right)\right),
\end{align*}
as desired.
\end{proof}

In Theorem~\ref{thm: prob}, we only considered a small \emph{central region} about $\lg n$, namely $\frac{3}{4}\lg n \le k \le 2 \lg n$, where the bulk of the distribution is concentrated.  For our subsequent analysis, we will also need restraints on the \emph{tails of the distribution}.  We do this with the following proposition.

\begin{proposition} The tails of the distribution of $L$ are exponentially small.  In particular, for $k<\frac{3}{4}\lg n$, we have 
\begin{equation}
\mathbb{P}_n(L< k)  = O(e^{-\sqrt[4]{n}/4}), \qquad n\rightarrow\infty,
\end{equation}
and for $k = 2\lg n + y$, we have
\begin{equation}
\mathbb{P}_n(L\ge 2 \lg n + y) = O\left(\frac{2^{-y}}{n}\right), \qquad n\rightarrow\infty,
\label{eq: right tail}
\end{equation}
uniformly for $y>0$.
\label{prop: tails}
\end{proposition}

\noindent In the right tail, we introduce the additional parameter of $y$ to establish that, when $k>2\lg n$, the probability $\mathbb{P}_n(L\ge  k)$ is not only exponentially small, but also \emph{uniformly decreasing} as $k\rightarrow\infty$.

\begin{proof}
For $k<\frac{3}{4}\lg n$, we find
$$\mathbb{P}_n(L< k) \le \mathbb{P}_n(L< \frac{3}{4}\lg n) = \frac{2}{\rho_k}(1-\rho_k)^2 e^{-n/(4n^{3/4})}\left(1 + O\left(\frac{\log n}{\sqrt{n}}\right)\right) = O(e^{-\sqrt[4]{n}/4}),$$
while for $k = 2\lg n + y$ with $y>0$, we find
\begin{align*}
\mathbb{P}_n(L\ge 2 \lg n + y) & \le 1 - \mathbb{P}_n(L<  2 \lg n + y)  \\ &
= 1 - \frac{2}{\rho_k}(1-\rho_k)^2 e^{-(2^{-y})/(4n)}\left(1+ O\left(\frac{\log n}{n}\right)\right) \\ &
 = 1 - O(e^{-(2^{-y})/(4n)})  \\ & 
 = O\left(\frac{2^{-y}}{n}\right).
\end{align*}
\end{proof}

\section{Mean and Variance}

Finally, we estimate the mean and variance of the distribution.  As suggested by Theorem~\ref{thm: prob}, the scale is $\lg n$ with some periodic fluctuations.  The fluctuation is of a kind frequently encountered in the analysis of algorithms (see \cite{Flajolet95mellintransforms}).  For instance, the equations for the mean and variance here bear a striking resemblance to those found by Archibald, Knopfmacher, and Prodinger in their analysis of the number of distinct letters in geometrically distributed sequences \cite{Archibald06distinctvalues}.
\begin{theorem} 
The expected value of $L$ satisfies 
$$\mathbb{E}_n(L) = \lg n + \frac{\gamma}{\log 2} - \frac{5}{2} + P(\lg n) + O\left(\frac{\log^2 n}{\sqrt{n}}\right),$$
where $P$ is a continuous periodic function whose Fourier expansion is given by
$$P(w) = -\frac{1}{\log 2} \sum_{k\in\mathbb{Z}\setminus \{0\}} \Gamma \left(\frac{2ik\pi}{\log 2}\right)e^{-2ik\pi w}.$$
\label{thm: expectation} 
\end{theorem}

\noindent The oscillating function $P(w)$ has mean value $0$ and is found to have tiny fluctuations, along the order of $10^{-5}$.

\begin{proof}
We have
$$\mathbb{E}_n(L) = \sum_{k \ge 1} \mathbb{P}_n(L \ge k) = \sum_{k \ge 1} \left(1 - \mathbb{P}_n(L < k)\right).$$
We evaluate the sum in three pieces: the two tails (Proposition~\ref{prop: tails}), and the central region where the distribution is concentrated (Theorem~\ref{thm: prob}).  For the left tail, i.e., $k<\frac{3}{4} \lg n$, we have $\mathbb{P}_n(L>k)$ exponentially close to $1$.  More precisely,
\begin{align}
\sum_{1 \le k < \frac{3}{4}\lg n} \left(1 - \mathbb{P}_n(L < k)\right) & = \sum_{1 \le k < \frac{3}{4}\lg n} \left(1 - O(e^{-\sqrt[4]{n}/4})\right) \nonumber \\ &
= \sum_{1 \le k < \frac{3}{4}\lg n} (1 - e^{-n/2^{k+2}}) + O\left( e^{-\sqrt[4]{n}/4} \log n  \right) + O\left( e^{-n/2^{k+2}} \log n \right) \nonumber \\ &
= \sum_{1 \le k < \frac{3}{4}\lg n} (1 - e^{-n/2^{k+2}})  + O\left(\frac{\log n}{\sqrt{n}}\right). 
\label{eq: sum in left tail}
\end{align}

For the right tail, where $k>2 \lg n$, we have $\mathbb{P}_n(L>k)$ is both exponentially small and uniformly decreasing as $k\rightarrow \infty$, as shown by Proposition~\ref{prop: tails}.  We can therefore write
$$\sum_{k > 2 \lg n} \mathbb{P}_n(L \ge k) = O\left(\mathbb{P}_n(L \ge 2 \lg n)\right) = O\left(\frac{1}{n}\right).$$
By a similar argument, we can also show that
\begin{equation}
\sum_{k > 2 \lg n}\left(1 - e^{-n/2^{k+2} }\right) = O(1/n).
\label{eq: sum in right tail}
\end{equation}
Finally, for the central region, we have
$$\mathbb{P}_n(L < k) = e^{-n/2^{k+2}} \left(1+O\left(\frac{\log n}{\sqrt{n}}\right)\right).$$
Hence
\begin{align}
\sum_{\frac{3}{4}\lg n \le k \le 2\lg n} \left(1 - \mathbb{P}_n(L < k)\right) &
 = \sum_{\frac{3}{4}\lg n \le k \le 2\lg n} \left(1 - e^{-n/2^{k+2} }\right) 
 + e^{-n/2^{k+2} }O\left(\frac{\log n}{\sqrt{n}}\right)O(\log n) \nonumber \\ & 
 =  \sum_{\frac{3}{4}\lg n \le k \le 2\lg n} \left(1 - e^{-n/2^{k+2} }\right) + O\left(\frac{\log^2 n}{\sqrt{n}}\right). \label{eq: sum in central region}
\end{align}

Combining equations (\ref{eq: sum in left tail}), (\ref{eq: sum in right tail}), and (\ref{eq: sum in central region}), we obtain
$$\mathbb{E}_n(L) = \sum_{k \ge 1} (1 - e^{-n/2^{k+2}})  + O\left(\frac{\log^2 n}{\sqrt{n}}\right),$$
which can be rewritten as
\begin{equation}
\mathbb{E}_n(L) = \Phi\left(\frac{n}{4}\right) - 1 +  O\left(\frac{\log^2 n}{\sqrt{n}}\right), \quad \text{where} \quad \Phi(x) = \sum_{h \ge 0} (1 - e^{-x/2^{h}}).
\label{eq: mean with phi}
\end{equation}

We can now obtain precise asymptotic information about $\Phi(x)$ through \emph{Mellin transform techniques} (see \cite{Flajolet95mellintransforms}).  The Mellin transform of $\Phi(x)$ is
$$ \Phi^*(s) := \int_0^\infty \Phi(x) x^{s-1} dx = - \frac{\Gamma(s)}{1-2^s}, \qquad \Re(s) \in (-1,0).$$ 
We see that $\Phi^*(s)$ has a double pole at $s=0$ and simple poles at $s = \frac{2ik\pi}{\log2}$, which indicate an asymptotic expansion of $\Phi(x)$ that involves a Fourier series.  We obtain
$$\Phi(x) = \lg x + \frac{\gamma}{\log 2} + \frac{1}{2} + P(\lg x) + O\left(\frac{1}{x}\right), \quad P(w) = -\frac{1}{\log 2} \sum_{k\in\mathbb{Z}\setminus \{0\}} \Gamma \left(\frac{2ik\pi}{\log 2}\right)e^{-2ik\pi w}.$$
Evaluating $\Phi(x)$ at $x = \frac{n}{4}$ and substituting into equation (\ref{eq: mean with phi}) gives the desired result.
\end{proof}

\begin{theorem}
The variance of $L$ satisfies 
$$\mathbb{V}_n(L) = \frac{1}{12} + \frac{\pi^2}{6\log^2 2} + \epsilon + O\left(\frac{\log ^4 n}{\sqrt{n}}\right),
$$
where $|\epsilon| < 10^{-4}$.
\end{theorem}

\begin{proof}
We use the identity
$$\mathbb{V}_n(L) = \mathbb{E}_n(L^2) - \mathbb{E}_n(L)^2,$$
and we start by computing the second moment of $L$.  The methods are similar to those used to compute $\mathbb{E}_n(L)$, and their presentation is abbreviated here.  We have
\begin{align*}
\mathbb{E}_n(L^2) & 
= \sum_{k=1}^{\infty} k^2 \mathbb{P}_n(L=k) = \sum_{k=1}^{\infty} (2k-1)  \mathbb{P}_n(L\ge k) \\ & 
= \sum_{k=1}^{\infty} (2k-1)(1 -  \mathbb{P}_n(L < k)) \\ &
= \sum_{k \ge 1}(2k-1)(1 - e^{-n/2^{k+2}}) + O\left(\frac{\log^3 n}{\sqrt{n}}\right) \\ &
= \Psi\left(\frac{n}{4}\right) + 1 +  O\left(\frac{\log^3 n}{\sqrt{n}}\right),
\end{align*}
where $\Psi(x) = \sum_{h \ge 0}(2h-1)(1 - e^{-x/2^{h}})$. 

The Mellin transform of $\Psi(x)$ is given by
$$\Psi^*(s) = -\Gamma(s)\sum_{h\ge 0} (2h-1)2^{sh} = -\frac{\Gamma(s)(1-3\cdot2^s)}{(1-2^s)^2}.$$
The transform has a triple pole at $s=0$ and double poles at $s= \frac{2ik\pi}{\log 2}$ for $k\in\mathbb{Z}\setminus\{0\}$, the singular expansions of which yield the following asymptotic form of $\Psi(x)$:
$$\Psi(x) = \lg^2 x + \lg x\left(\frac{2\gamma}{\log 2} - 1 + 2P(\lg x)\right) - \frac{2}{3} + \frac{\pi^2 + 6\gamma^2}{6\log^2 2} - \frac{\gamma}{\log 2} -P(\lg x) + Q(\lg x) + O\left(\frac{1}{x}\right),$$
where $P(w)$ is as in the statement of Theorem~\ref{thm: expectation} and $Q(w)$ is a periodic function with Fourier expansion
$$Q(w)  = \frac{2}{\log^2 2} \sum_{k\in\mathbb{Z}\setminus\{0\}} \psi\left(\frac{2ik\pi}{\log2}\right)\Gamma\left(\frac{2ik\pi}{\log2}\right)e^{-2ik\pi w}.$$

We therefore have
\begin{align}
\mathbb{E}_n(L^2) & = \lg^2 n + \lg n\left(\frac{2\gamma}{\log 2} - 5 + 2P(\lg n)\right) + \frac{19}{3} + \frac{\gamma^2}{\log^2 2} + \frac{\pi^2}{6\log^2 2} \nonumber \\ &
 - \frac{5\gamma}{\log 2} - 5P(\lg n) + Q(\lg n) + O\left(\frac{\log^3 n}{\sqrt{n}}\right). \label{eqn: E(L2)}
\end{align}
Meanwhile, the square of $\mathbb{E}_n(L)$ satisfies
\begin{align}
\mathbb{E}_n(L)^2 & = \lg^2 n + \lg n\left(\frac{2\gamma}{\log2} - 5 + 2P(\lg n)\right) + \frac{25}{4}
+ \frac{\gamma^2}{\log^2 2}  \nonumber \\ & 
- \frac{5\gamma}{\log 2} - 5P(\lg n) + \frac{2\gamma}{\log 2}P(\lg n)  + P(\lg n)^2+ O\left(\frac{\log ^4 n}{\sqrt{n}}\right). \label{eqn: E(L) square}
\end{align}
Subtracting equation (\ref{eqn: E(L) square}) from equation (\ref{eqn: E(L2)}), we obtain
\begin{align*}
\mathbb{V}_n(L) & = \frac{1}{12} + \frac{\pi^2}{6\log^2 2} + Q(\lg n)- \frac{2\gamma}{\log 2}P(\lg n) - P(\lg n)^2 +  O\left(\frac{\log ^4 n}{\sqrt{n}}\right) \\ &
=  \frac{1}{12} + \frac{\pi^2}{6\log^2 2} + \epsilon + O\left(\frac{\log ^4 n}{\sqrt{n}}\right),
\end{align*}
where $|\epsilon| < 10^{-4}$.
\end{proof}

\begin{figure}

\begin{center}
\Img{2.0}{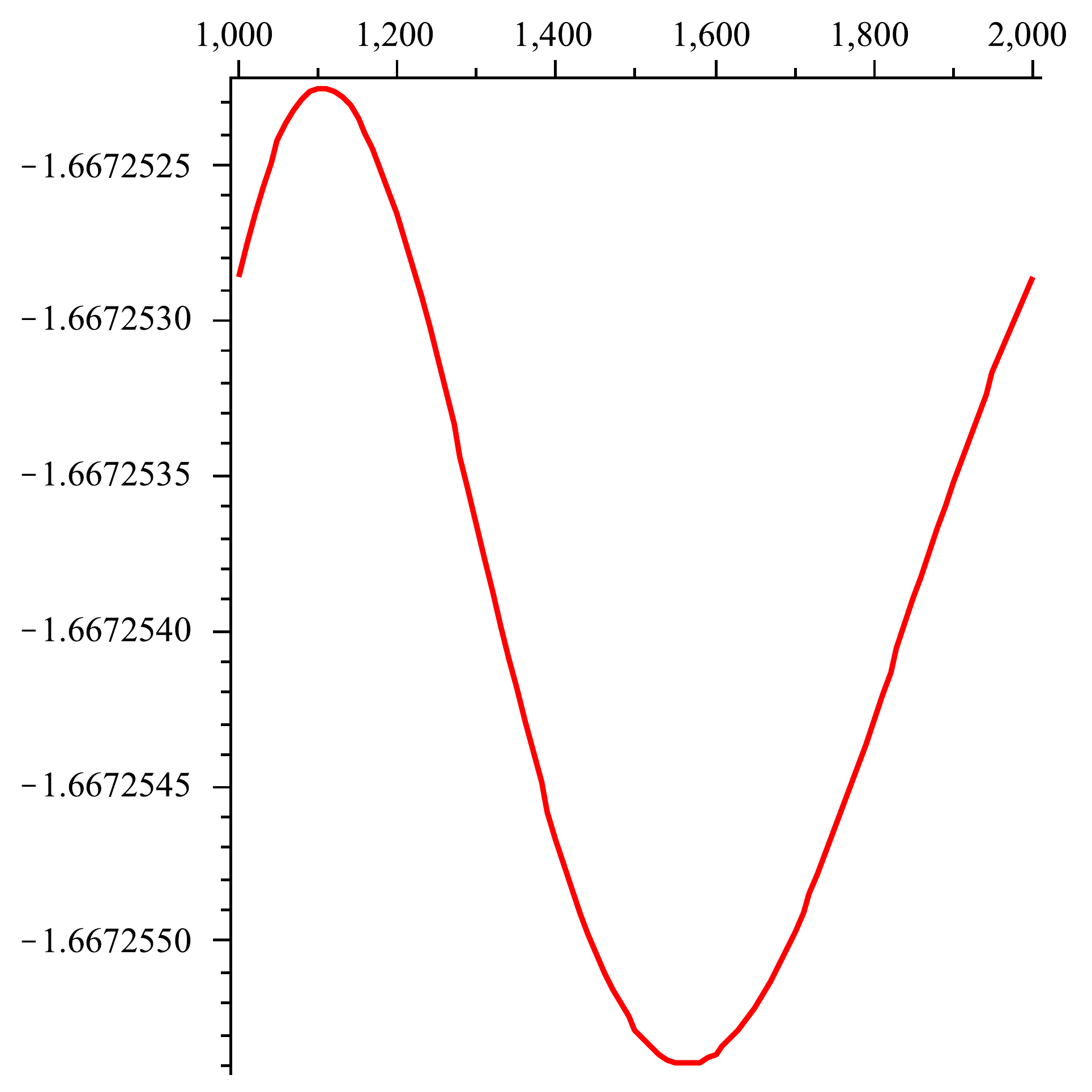}\hspace*{0.5truein}\Img{2.0}{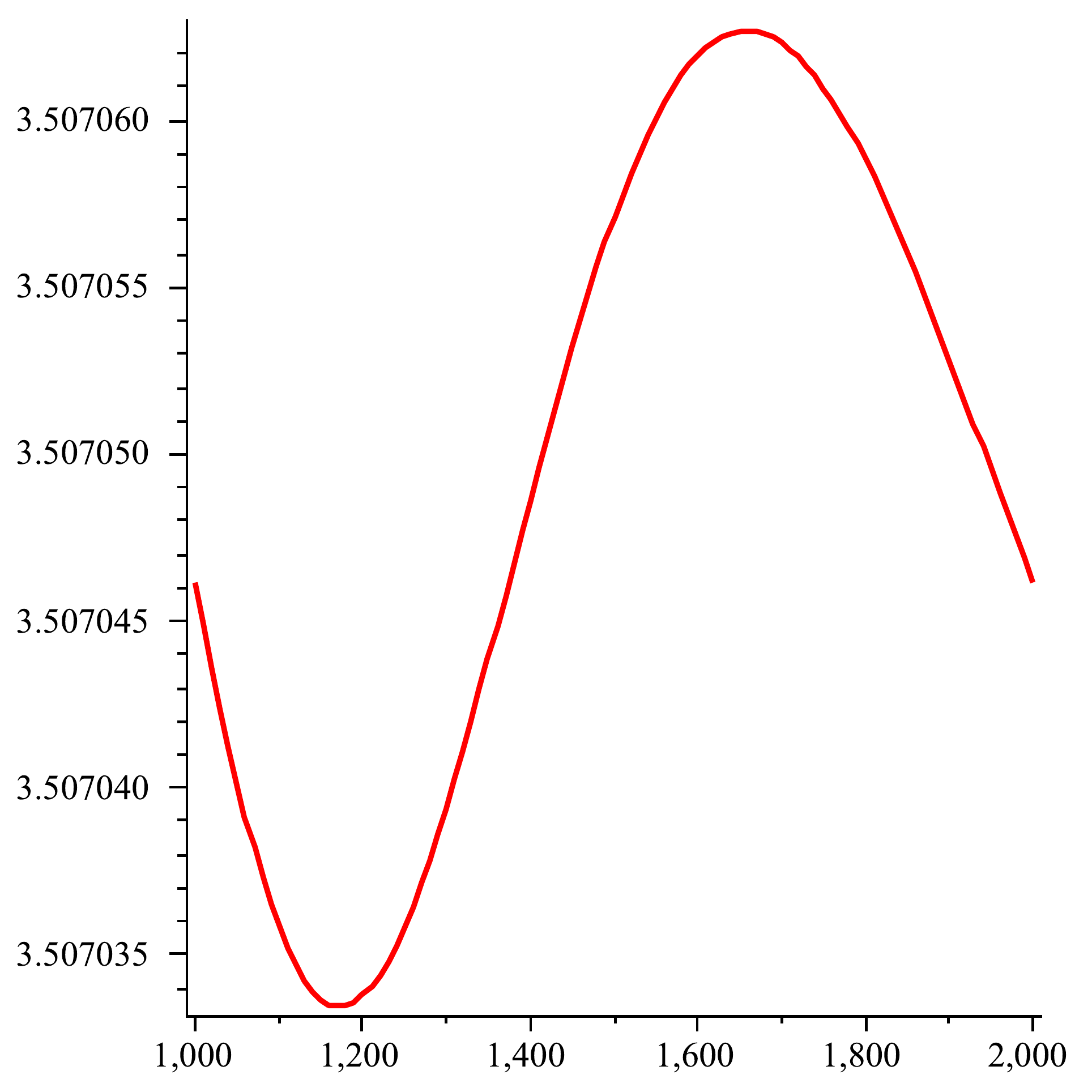}
\end{center}

\caption{\label{periodic-fig}
The functions
${\Phi\left(\frac{x}{4}\right) - \lg x - 1}$ (left) and ${\Psi\left(\frac{x}{4}\right)+ 1 - \left(\Phi\left(\frac{x}{4}\right)-1\right)^2}$ (right), which represent the ``constant parts'' in the asymptotic form of  the mean and variance of $L$.
}

\end{figure}

Figure~\ref{periodic-fig} displays the periodic functions associated with the mean and the variance.  We see that the amplitudes of the fluctuations in the periodic functions are very small, about $1.6\times10^{-6}$ in the mean and $1.5\times10^{-5}$ in the variance.

\section{Conclusion}

It is well-known that the total number of summands in a random integer composition of size $n$ is $\sim n/2$ both on average and in probability.  Furthermore, the number of summands equal to $1,2,3\ldots$ is close to $\frac{n}{4}, \frac{n}{8}, \frac{n}{16},\ldots$, respectively \cite[p.~168]{Flajolet09analyticCombinatorics}.  Using this elementary fact, the longest run problem could alternatively be approached by studying words of length $n/2$ and looking at the longest run of 1's, then the longest run of 2's, and so on, and combining the results.  To do so, we can make use of the fact that in a  word of length $\nu$ over the binary alphabet $\{a,b\}$, where $a$ occurs with relative frequency $p$, the expected length of the longest run of $a$'s is 
$$k_0 = \log_{1/p}\nu.$$
In a random composition of size $n$, the summand $r$ occurs with relative frequency approximately $1/2^r$, so we would expect the longest run of $r$'s to be roughly ($\nu\rightarrow n/2$)
$$k_{0}(r) \sim \log_{2^r}\frac{n}{2} \sim \frac{\lg n}{r}.$$

This heuristic approach can be made rigorous by analytic methods.  One would alter the generating functions to record the longest run of a particular summand $r$ in the composition, and then carry out analysis of the dominant pole $\rho_k(r)$.  The important observation is that the exponent of $\frac{1}{2}$ in the approximation of $\rho_k(r)$ translates directly to the coefficient of $\lg n$ in the expectation of the longest run of $r$'s.  The technical details are entirely similar to the analysis in this note.  We obtain:

\begin{proposition}
Let $L_r$ be a random variable representing the longest run of $r$'s in a composition.  Then we have 
\begin{equation}
\mathbb{E}_n(L_r) = \frac{1}{r} \lg n + O(1).
\label{eq: heuristic}
\end{equation}
\end{proposition}

\begin{figure}

\begin{center}
\Img{1.5}{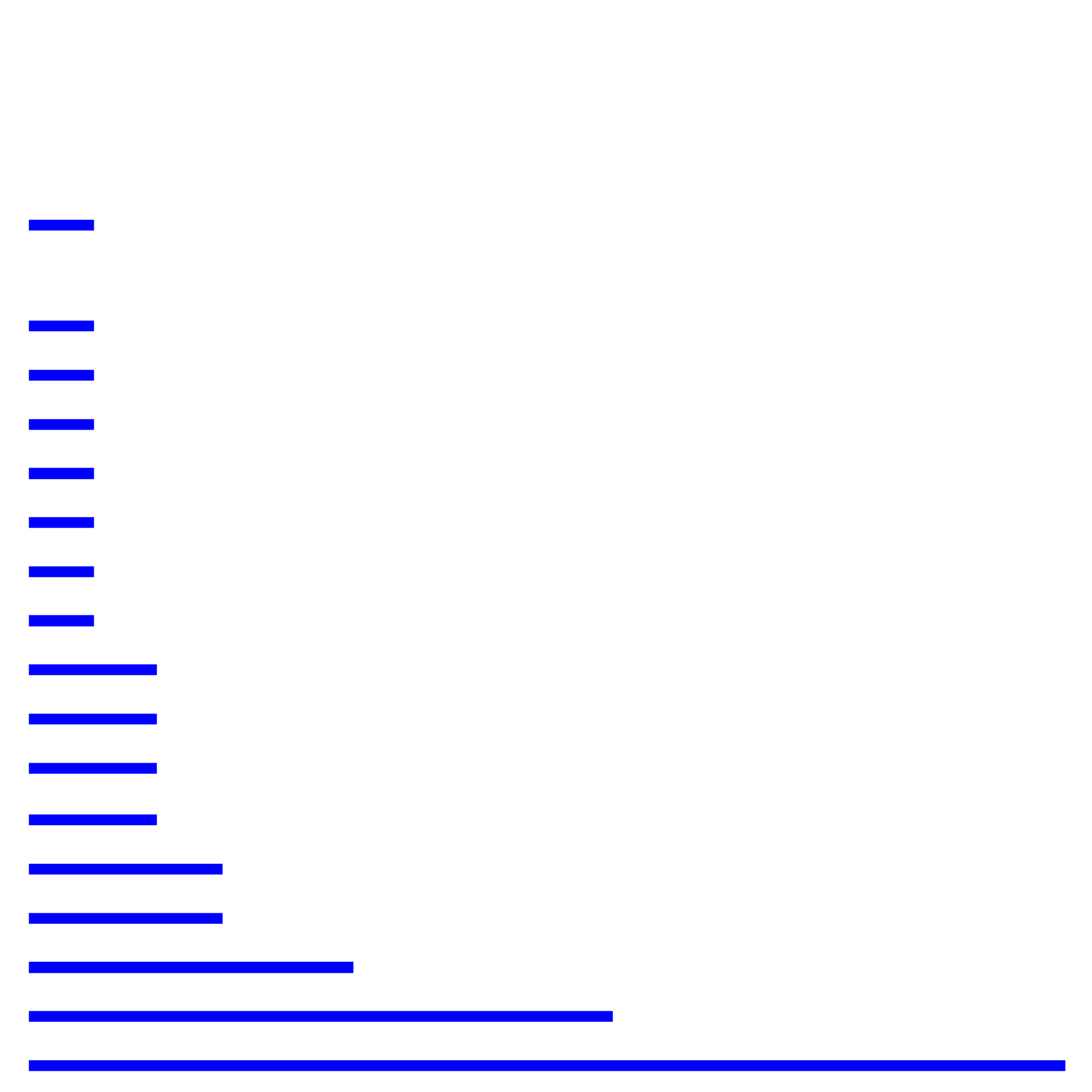} \hspace{.5in} \Img{1.5}{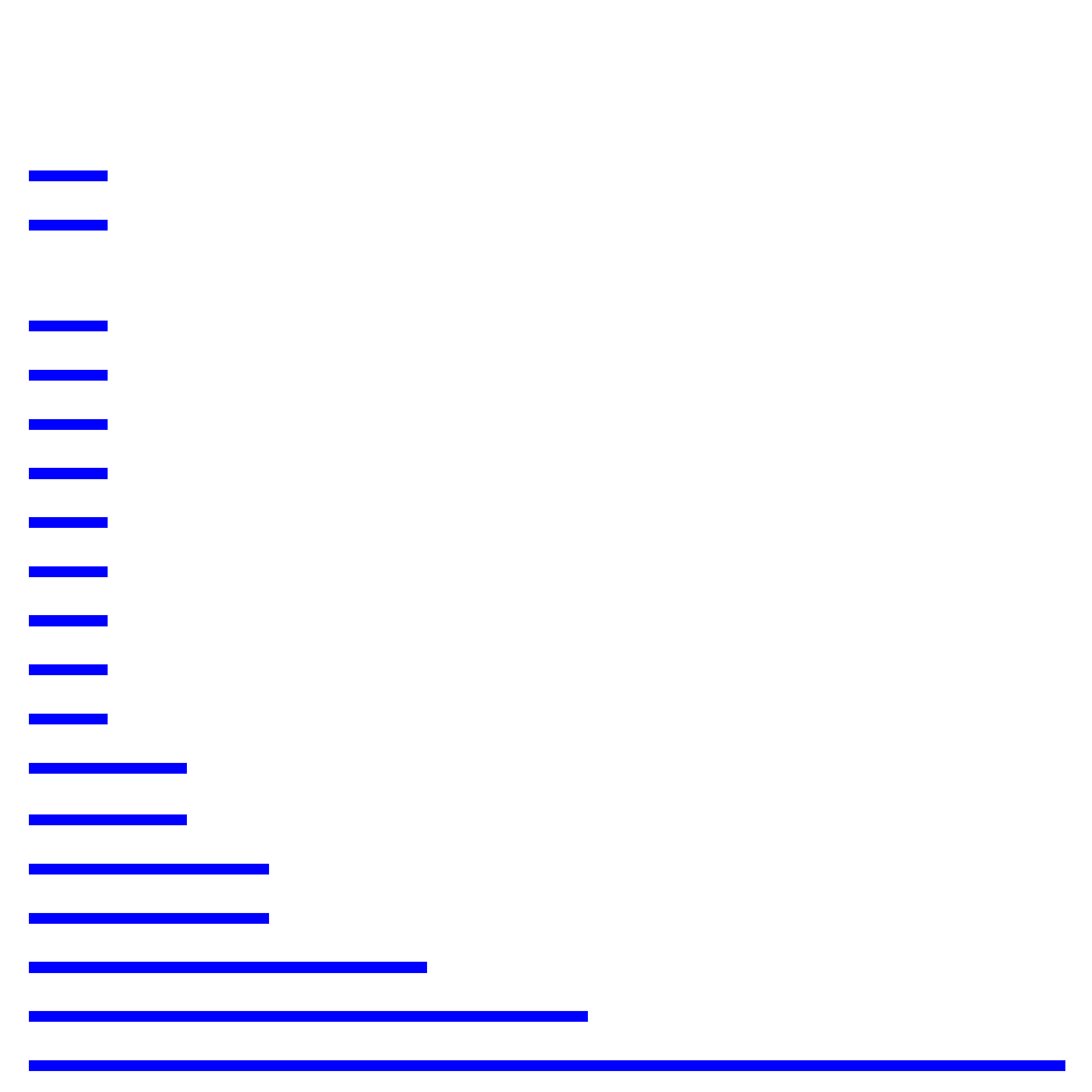} \\ \Img{1.5}{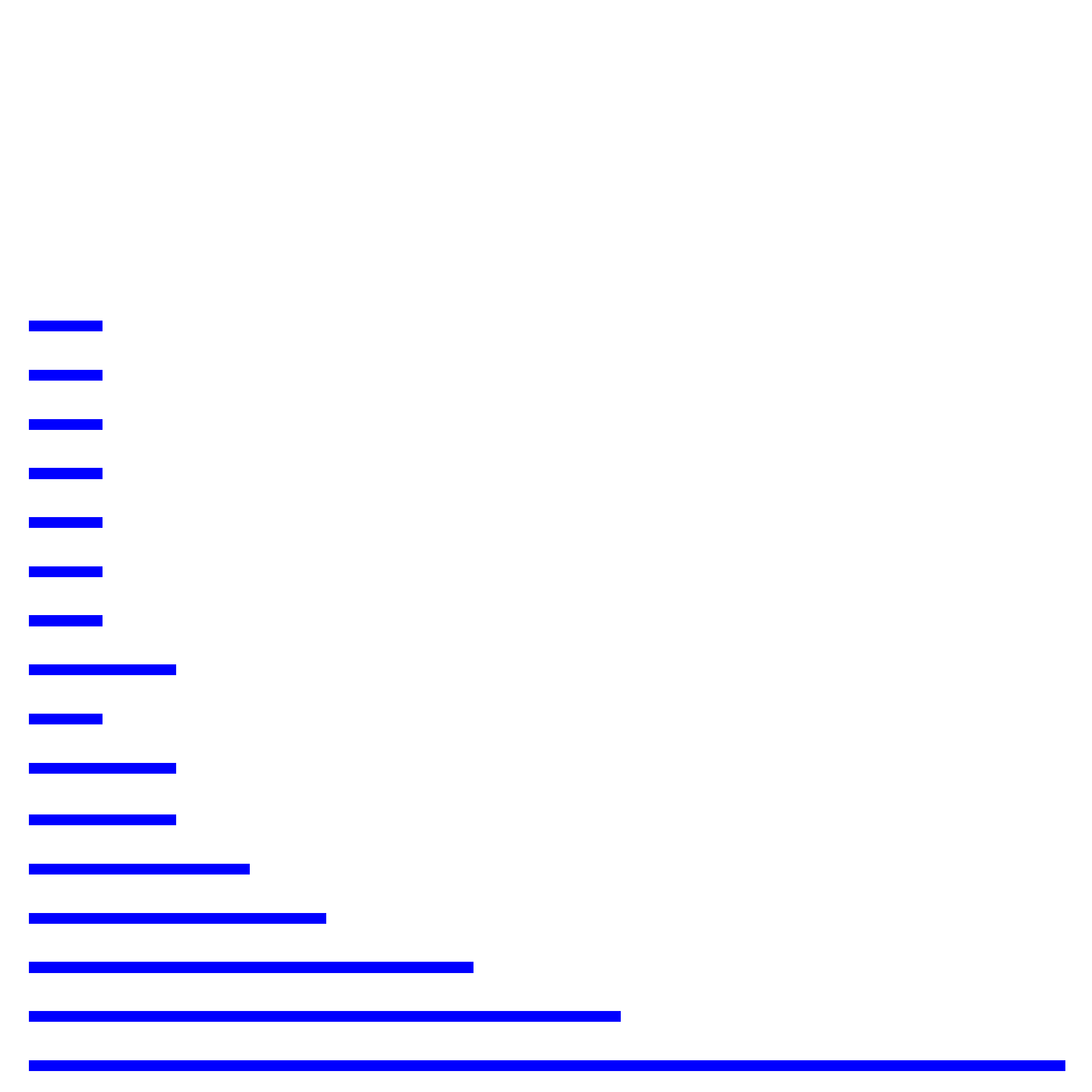} \hspace{.5in} \Img{1.5}{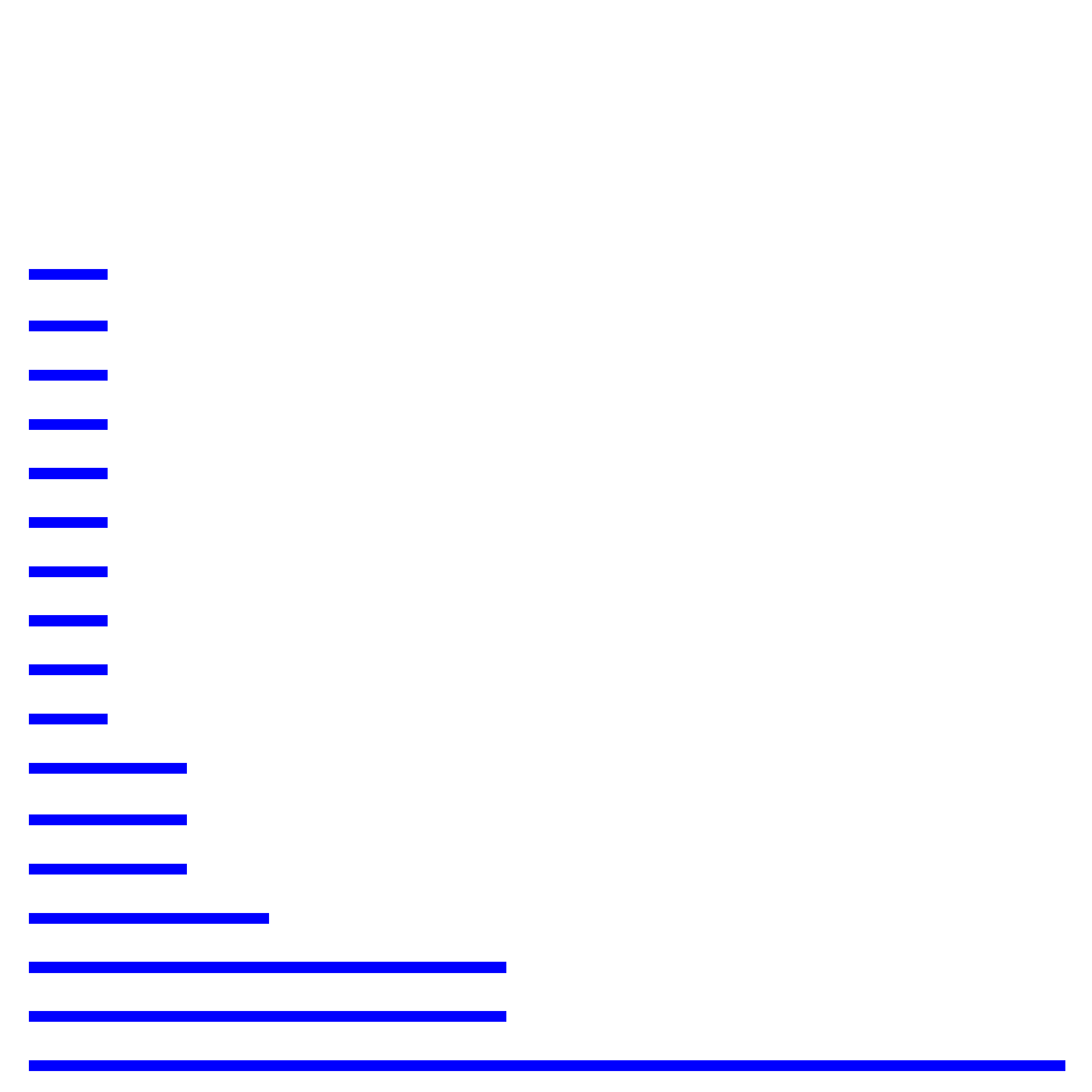}
\end{center}

\caption{\label{fig: largest summand}
The longest run, $L_r$, of each integer $r$ in four compositions of size $10^5$.  The integers are represented in increasing order from bottom to top (the bottom line represents $L_1$).
}
\end{figure}

Again, the distribution is highly concentrated around the mean.  We then see that $L\sim L_1$ with high probability.  In fact, since the length of the longest run of 1's is likely to be twice that of 2's, and since the $L_r$ have small variance, we should expect the longest run of equal parts to almost certainly be a run of $1$'s.  We can even infer that the constant term of $\mathbb{E}_n(L_1)$ will coincide with that of $\mathbb{E}_n(L)$, given in Theorem~\ref{thm: expectation}.  This intuition is supported by Figure~\ref{fig: largest summand}, which displays the graphs of $L_r$ for four simulations of random integer compositions of size $10^5$.  We see that $L_1$ is much greater than $L_2$ in all cases.  

Equation \ref{eq: heuristic}, though only valid for fixed $r$, suggests that the largest summand (i.e., the largest $r$ such that $L_r\ge1$) should be close to $\lg n$, which is a true fact \cite[p.~169]{Flajolet09analyticCombinatorics}.  This is also illustrated by Figure~\ref{fig: largest summand}, where we see that $L_r$ is strictly positive until around $r \sim \lg n$, when it alternates sporadically between 0 and 1 until it ultimately peters out completely.  Such behavior suggests that in a random composition of size $n$, all summands up to $\lg n + O(1)$ will be present with high probability.  Meanwhile, summands at $\lg n + \omega(n)$, for any $\omega(n) \rightarrow\infty$, will likely not occur at all.  This is consistent with the analysis of Archibald, Knopfmacher, and Prodinger in \cite{Archibald06distinctvalues}, where the number of distinct values in a geometrically distributed sequence was found to be $\lg n +O(1)$.  

In our analysis, we found a periodic fluctuation in the distribution which is common to several problems in analysis of algorithms.  These common fluctuations are not coincidence, but rather the result of an underlying structure of the singularities, namely a geometric displacement of a fixed pole $\rho_k$ with respect to $1/2$ (see Proposition \ref{prop: rho}).  This structure induces the same asymptotic form in the expectations of several values, which on the surface seem to be unrelated (see \cite{Flajolet95mellintransforms}).

%\bibliographystyle{plain}
%\bibliography{ayla}

\end{document}